\documentclass{article}

\usepackage[all]{xy}
 \usepackage[usenames,dvipsnames]{pstricks}
 \usepackage{epsfig}
 \usepackage{pst-grad} % For gradients
 \usepackage{pst-plot} % For axes
\usepackage{ebezier}

\usepackage[
        colorlinks, citecolor=red,
        backref,
        pdfauthor={MDJ,JK},
        pdftitle={S1}
]{hyperref}

\usepackage{latexsym}\usepackage{amsfonts}\usepackage{amssymb}\usepackage{amsthm}
\usepackage{amsmath}\usepackage{newlfont}\usepackage{graphicx}\usepackage{doc}
\usepackage{color}\usepackage{multicol}\usepackage{epic}\usepackage{eepic}
\usepackage{ebezier}
\newtheorem{thrm}{Theorem}

\newtheorem{cor}[thrm]{Corollary}

\newtheorem{rem}[thrm]{Remark}
\newtheorem{prop}[thrm]{Proposition}
\newtheorem*{theorem-non}{Theorem}

\newtheorem*{namedthm}{\namedthmname}
\newcounter{namedthm}



\include{epsf}

\def \Dj{\mbox{\raise0.3ex\hbox{-}\kern-0.4em D}}

\begin{document}

\title{Polynomial entropy on regular curves}

\author{Ma\v{s}a \Dj ori\'c\thanks{Corresponding author: Ma\v{s}a \Dj ori\'c.}\\
Matemati\v{c}ki institut SANU\\
Knez Mihailova 36\\
11000 Beograd\\
Serbia\\
masha@mi.sanu.ac.rs\\ \and
Jelena Kati\'c
\thanks{This work is partially supported by the Ministry of Education, Science and Technological Developments of Republic of Serbia: grant number 451-03-47/2023-01/200104 with Faculty of Mathematics, University of Belgrade.} \\
Matemati\v cki fakultet\\
Studentski trg 16\\
11000 Beograd\\
Serbia\\
jelena.katic@matf.bg.ac.rs
}

\maketitle

\begin{abstract} We show that the polynomial entropy of homeomorphisms on regular curves is bounded above by one. Moreover, the polynomial entropy equals one under the fairly mild condition that the homeomorphism possesses a wandering point. We obtain a rigidity result for homeomorphisms on local dendrites (and therefore on graphs and dendrites): their polynomial entropy is either zero or one.
\end{abstract}

\medskip

\section{Introduction}

One useful measure of complexity for systems with zero topological entropy is polynomial entropy. The very definitions of topological and polynomial entropy differ in that the former quantifies the overall {\it exponential} growth of orbit complexity, while the latter captures the growth rate on a {\it polynomial} scale. Polynomial entropy serves as a refined tool to distinguish among systems with zero topological entropy. For instance, consider a rotation on the circle versus a homeomorphism with both periodic and wandering points. Although both have topological entropy equal to zero, the first system is evidently simpler than the second. Labrousse proved that polynomial entropy can distinguish between such systems (see Theorem 1 in~\cite{L}).

Topological and polynomial entropy share several properties: both are conjugacy invariants, they depend only on the topology (not on the specific choice of metric), they satisfy the finite union property, and they both obey a product formula. However, there are important differences as well. Certain properties such as the power formula, the $\sigma$-union property, and the variational principle hold for topological entropy but do not necessarily hold for polynomial entropy (see~\cite{M1,M2}). Moreover, topological entropy is equal to the entropy of the system restricted to the non-wandering set, which is a closed, invariant subset. This is not the case for polynomial entropy. Unlike topological entropy, which captures only the behavior on the non-wandering set, polynomial entropy is sensitive to the wandering part of the system.

Polynomial entropy of homeomorphisms on the interval $[0,1]$ and circle $\mathbb{S}^1$ have been computed, as well as polynomial entropy on some dendrites. In this paper, we obtain an upper bound for polynomial entropy on regular curves, a class of one dimensional Peano continuums which contains dendrites, graphs and more generally local dendrites. 

Let $\varphi(f,n)=\sup\limits_{x\in X}\{|\mathrm{Comp}(f^{-n}(x))|\}.$ Then it holds:

\begin{theorem-non}
   Let $f:X\to X$ be a continuous on a regular curve $X$. Then $$h_{\mathrm{pol}}(f)\leqslant1+\limsup_{n\to\infty}\dfrac{\log{\varphi(f,n)}}{\log{n}}.$$
\end{theorem-non}

We apply this theorem to the $n$-fold symmetric product $F_n(X)$ and $2^X$. We also derive an upper bound for polynomial entropy on the interval $[0,1]$, using the lap number, an estimate already proved by Gomes and Carneiro in~\cite{GC}.

Roth, Roth and Snoha obtained a flexibility result concerning polynomial entropy on dendrites: they constructed a dendrite and surjective continuous map on it, with the given polynomial entropy $\alpha\in (2,3)$ (see~\cite{RRS}). Here we obtain a rigidity result, as a generalisation of a result of Labrousse~\cite{L} for interval and circle homeomorphisms.

\begin{theorem-non}
 Let $X$ be a local dendrite and $f:X\to X$ a homeomorphism. Then $h_{\mathrm{pol}}(f)\in\{0,1\}$. Moreover $h_{\mathrm{pol}}(f)=1$ if and only if $f$ possesses a wandering point.
\end{theorem-non}

\section{Preliminaries}

\subsection{Regular curves}\label{subsec:reg}

We say $D$ is a \textit{dendrite} if $D$ is a locally connected continuum, i.e. a nonempty connected compact metric space, which contains no simple closed curves. A point $x\in D$ is an \textit{end point} of $D$ if $D\backslash\{x\}$ is connected. $x\in D$ is a \textit{cut point} of $D$ if $D\backslash\{x\}$ is not connected. We denote the number of components of $D\backslash\{x\}$ by $\mathrm{ord}(x)$. If the order of $x$ is equal to $2$, then $x$ is an \textit{ordinary point} of $D$. If the order of $x$ is greater than $2$, then $x$ is a \textit{branch point} of $D$. We denote by $B(D)$ the set of all branch points and by $E(D)$ the set of all endpoints of $D$. Some basic properties of dendrites and dendrite maps can be found, for example, in \cite{JB}.

A \textit{local dendrite} is a continuum such that every point has a dendrite neighbourhood. Every local dendrite is a \textit{Peano continuum} (locally connected continuum) which has a finite number of circles. Dendrites and graphs are local dendrites.

Let $X$ be a continuum. We say that $X$ is a \textit{regular curve} if for each $x\in X$ and each open neighbourhood $U$ of $x$ in $X$ there exists an open neighbourhood $V$ of $x$ in $X$ such that $V\subset U$ and the boundary set $\partial V$ is finite. It is known that Peano continua are arcwise connected and locally arcwise connected (see Theorems 8.23 and 8.25 in \cite{SN}). Every regular curve is a Peano one dimensional continuum, and therefore is arcwise and locally arcwise connected. Graphs, dendrites and, more generally, local dendrites are regular curves (for more details see \cite{SN}). The proof of the following theorem can be found in \cite{GS}.

\begin{prop}
If $f:X\to X$ is a homeomorphism on a regular curve $X$, then the topological entropy of $f$ is equal to zero.
\end{prop}

\subsection{Polynomial entropy}\label{subsec:entropy}

Suppose that $X$ is a compact metric space, and $f:X\rightarrow X$ is continuous. Denote by $d_n^f(x,y)$ the dynamic metric (induced by $f$ and $d$):
$$
d_n^f(x,y)=\max\limits_{0\leq k\leq n-1}d(f^k(x),f^k(y)).
$$
Fix $Y\subseteq X$. For $\varepsilon>0$, we say that a finite set $E\subset X$ is \textit{$(n,\varepsilon)$-separated} if for every $x,y \in E$ it holds $d_n^f(x,y)\geq \varepsilon$. Let $\mathrm{sep}(n,\varepsilon;Y)$ denote the maximal cardinality of an $(n,\varepsilon)$-separated set $E$, contained in $Y$.

\begin{def}\label{def:pol_ent} The \textit{polynomial entropy} of the map $f$ on the set $Y$ is defined by
$$
h_{\mathrm{pol}}(f;Y)=\lim\limits_{\varepsilon \rightarrow 0}\limsup\limits_{n\rightarrow \infty}\frac{\log \mathrm{sep}(n,\varepsilon;Y)}{\log n}.
$$
\end{def}

We can also define the polynomial entropy as follows. Let $\mathrm{span}(n,\varepsilon;Y)$ denote the minimal number of balls of radius $\varepsilon$ (with respect to $d_n^f$) that cover $Y$ (the centers of the ball are not necessarily elements of $Y$). Denote by $\mathrm{cov}(n,\varepsilon;Y)$ the minimal number of sets $Y_j$ such that the diameters (with respect to $d_n^f$) of $Y_j$ are smaller than $\varepsilon$ and $Y\subseteq\cup_{j=1}^mY_j$.

From the following sequence of inequalities
$$\mathrm{span}(n,2\varepsilon;Y)\le \mathrm{cov}(n,\varepsilon;Y)\le \mathrm{sep}(n,\varepsilon;Y)\le \mathrm{cov}(n,\varepsilon/2;Y)\le \mathrm{span}(n,\varepsilon/2;Y)$$ we conclude that
$$h_{\mathrm{pol}}(f;Y)=\lim\limits_{\varepsilon \rightarrow 0}\limsup\limits_{n\rightarrow \infty}\frac{\log \mathrm{span}(n,\varepsilon;Y)}{\log n}=\lim\limits_{\varepsilon \rightarrow 0}\limsup\limits_{n\rightarrow \infty}\frac{\log \mathrm{cov}(n,\varepsilon;Y)}{\log n}.$$

If $X=Y$ we abbreviate $h_{\mathrm{pol}}(f):=h_{\mathrm{pol}}(f;X)$, $\mathrm{span}(n,\varepsilon):=\mathrm{span}(n,\varepsilon;X)$ etc. We list some properties of the polynomial entropy that are important for our computations (for proofs see Propositions $1-4$ in \cite{M2}):
\begin{itemize}
\item[(1)] $h_{\mathrm{pol}}(f^k)=h_{\mathrm{pol}}(f)$, for any $k\geqslant 1$.
\item[(2)] If $Y\subset X$ is a closed, $f$-invariant set, then $h_{\mathrm{pol}}(f;Y)=h_{\mathrm{pol}}(f|_Y)$.
\item[(3)] If $Y=\bigcup_{j=1}^mY_j$ where $Y_j$ are $f$-invariant, then $h_{\mathrm{pol}}(f;Y)=\max\{h_{\mathrm{pol}}(f;Y_j)\mid j=1,\ldots,m\}$\label{finite-union}.\label{(3)}
\item[(4)] If $f:X\to X$, $g:Y\to Y$ and $f\times g:X\times Y\to X\times Y$ is defined as $f\times g (x,y):=(f(x),g(y))$, then $h_{\mathrm{pol}}(f\times g)=h_{\mathrm{pol}}(f)+h_{\mathrm{pol}}(g)$.
\item[(5)] $h_{\mathrm{pol}}(f)$ does not depend on a metric but only on the induced topology.
\item[(6)] $h_{\mathrm{pol}}(\cdot)$ is a \textit{conjugacy invariant} (meaning if $f:X\to X$, $g:X'\to X'$, $\varphi:X\to X'$ is a homeomorphism of compact spaces and $g\circ\varphi=\varphi\circ f$, then $h_{\mathrm{pol}}(f)=h_{\mathrm{pol}}(g)$).
\item[(7)] If $f:X\to X$ and $g:X'\to X'$ are \textit{semi-conjugated}, meaning that $\varphi:X\to X'$ is a continuous surjective map of compact spaces and $g\circ\varphi=\varphi\circ f$, then $h_{\mathrm{pol}}(f)\geqslant h_{\mathrm{pol}}(g)$.
\end{itemize}

A set $A\subset X$ is \textit{wandering} if $f^n(A)\cap A=\emptyset$, for all $n\geqslant 1$. A point $p\in X$ is \textit{wandering} if there exists a wandering neighbourhood $U\ni p$. \label{wan-pt}

A point that is not wandering is said to be \textit{non-wandering}. We denote the set of all non-wandering points by $\mathrm{NW}(f)$. The set $\mathrm{NW}(f)$ is closed and $f$-invariant. Also, we denote the set of all fixed points by $\mathrm{Fix}(f)$ and the set of all periodic points by $\mathrm{Per}(f)$.

Recall that $f:X\to X$ is equicontinuous if for all $\varepsilon>0$ there is a $\delta>0$ such that if $x,y\in X$, $d(x,y)<\delta$, then $d(f^n(x),f^n(y))<\varepsilon$, for all $n\in\mathbb{Z}$. For an equicontinuous homeomorphism on a compact metric space there is a compatible metric that makes it an isometry - $d'(x,y)=\sup_{n\in\mathbb{Z}}d(f^n(x),f^n(y))$. Therefore, we can conclude that the following holds
\begin{prop}\label{equi}
Let $f:X\to X$ be an equicontinuous homeomorphism on a compact metric $X$. Then $h_{\mathrm{pol}}(f)=0$.
\end{prop}

\subsection{Hyperspaces and induced maps}

For a compact metric space $(X,d)$, the hyperspace $2^X$ is the set of all nonempty closed subsets of $X$. The topology on $2^X$ is induced by the Hausdorff metric
$$d_H(A,B):=\inf\{\varepsilon>0\mid A\subset U_\varepsilon(B),\;B\subset U_\varepsilon(A)\},$$
where
\begin{equation}\label{eq:neighb}
U_\varepsilon(A):=\{x\in X\mid d(x,A)<\varepsilon\}.
\end{equation}
We consider a closed subspace $F_n(X)$ of $2^X$, $n\geqslant 1$, consisting of all finite nonempty subsets of cardinality at most $n$, with the induced metric. The set $F_n(X)$ is called the \textit{$n$-fold symmetric product of $X$}. Both spaces $2^X$ and $F_n(X)$ are compact metric spaces with respect to the Hausdorff metric.

If $f:X\to X$ is continuous, then it induces continuous maps
 $2^f:2^X\to 2^X$ and $F_n(f):F_n(X)\to F_n(X)$, by:
 $$2^f(A):=\{f(x)\mid x\in A\}.$$
 If $f$ is a homeomorphism, then this is also true for both $2^f$ and $F_n(f)$.

Let $X^{\times n}$ denote $X\times\ldots\times X$. For a dynamical system $f$ on $X$, the induced symmetric product map $F_n(f)$ is a factor of the induced product system $f\times\ldots\times f$, denoted by $f^{\times n}$ and defined by:
$$f^{\times n}:X^{\times n}\to X^{\times n},\quad f^{\times n}(x_1,\ldots,x_n):=(f(x_1),\ldots,f(x_n)).$$
Indeed, if we define
$$\pi_n:X^{\times n}\to F_n(X),\quad \pi_n:(x_1,\ldots,x_n)\mapsto\{x_1,\ldots,x_n\},$$ one easily checks that $\pi_n$ is onto, continuous and that $\pi_n\circ f^{\times n}=F_n(f)\circ\pi_n$. Therefore, it always holds
$$h_{\mathrm{pol}}(F_n(f))\leqslant h_{\mathrm{pol}}(f^{\times n}).$$

\section{Bounds for polynomial entropy on regular curves}

Following the results from Kato's paper \cite{HK} for topological entropy, we can obtain an upper bound for polynomial entropy of a continuous map on a regular curve. We will give a brief sketch of the proof and formulate our first main result. \\
Let $X$ be a compact metric space. Then $\mathrm{Comp}(X)$ denotes the set of all connected components of $X$. As per usual, $|A|$ denotes the cardinality of the set $A$. Let $f:X\to X$ be a continuous map on a compact space $X$ and $n\in\mathbb{N}$. We define:
$$\varphi(f,n)=\sup_{x\in X}\{|\mathrm{Comp}(f^{-n}(x))|\}.$$
Note that if $f:X\to X$ is a homeomorphism, then $\varphi(f,n)=1$. We can now formulate our main result:
\begin{thrm}\label{kato}
  Let $f:X\to X$ be a continuous on a regular curve $X$. Then $$h_{\mathrm{pol}}(f)\leqslant1+\limsup_{n\to\infty}\dfrac{\log{\varphi(f,n)}}{\log{n}}.$$
\end{thrm}
\begin{proof} For $\varepsilon>0$, choose a finite open cover $\{U_1,\ldots,U_k\}$ such that $\mathrm{diam}(U_j)<\varepsilon$ and the boundary $\partial(U_j)$ is a finite set for each $j$ (this cover exists since $X$ is a regular curve). Let 
$$A_1:=\overline{U}_1,\quad A_{j+1}:=\overline{U_{j=1}}\setminus\bigcup _{i+1}^jU_j,\;\mbox{for}\; j\in\{2,\ldots,k\}.$$ The set $\mathcal{A}:=\{A_1,\ldots,A_k\}$ is a finite closed cover of $X$ such that $\partial(A_j)$ is finite for every $j$. Denote by $L$ the cardinality of the set
$$\bigcup\{\partial(A)\mid A\in\mathcal{A}\}.$$
In the proof of Theorem 2.1 in~\cite{HK}, H. Kato constructed an $(n,\epsilon)$-spanning set with cardinality less than or equal to $L\cdot\sum\limits_{i=0}^{n-1}\varphi(f,i)$, for any $n\in\mathbb{N}$. Note that the constant $L$ does not depend on $n$, but only on $\varepsilon$. For each $n$ choose $i_n$, $0\leqslant i_n\leqslant n-1$ such that $\max{\{\varphi(f,i)\mid i=0,1,\ldots,n-1\}}=\varphi(f,i_n)$. Following the notation from the mentioned Kato's paper let $r(\varepsilon)=\limsup\limits_{n\to\infty}\frac{\mathrm{span}(n,\varepsilon)}{\log{n}}$. Then we have

\begin{align*}
r(\varepsilon)&=\limsup\limits_{n\to\infty}\frac{\log\mathrm{span}(n,\varepsilon)}{\log{n}}\\
&\leqslant\limsup\limits_{n\to\infty}\frac{\log{\left(n\cdot L\cdot \max{\{\varphi(f,i)\mid i=0,1,\ldots,n-1\}}\right)}}{\log{n}}\\
&=\limsup\limits_{n\to\infty}\frac{\log{n}+\log{L}+\log{\varphi(f,i_n)}}{\log{n}}\\
&=1+\limsup\limits_{n\to\infty}\frac{\log{\varphi(f,i_n)}}{\log{n}}\\
&\leqslant1+\limsup\limits_{n\to\infty}\frac{1}{\log{i_n}}\log{\varphi(f,i_n)}\\
&\leqslant 1+\limsup\limits_{n\to\infty}\frac{\log{\varphi(f,n)}}{\log{n}}.
\end{align*}

We can conclude $h_{\mathrm{pol}}(f)=\lim\limits_{\varepsilon\to0}r(\varepsilon)\leqslant1+\limsup\limits_{n\to\infty}\dfrac{\log{\varphi(f,n)}}{\log{n}}$.
  
\end{proof}

\begin{cor}
   Let $f:X\to X$ be a homeomorphism on a regular curve $X$. Then $h_{\mathrm{pol}}(f)\leqslant1$.
\end{cor}
\begin{proof}
As we already mentioned, it is clear that when $f:X\to X$ is a homeomorphism, then $\varphi(f,n)=1$. We immediately get:
$$h_{\mathrm{pol}}(f)\leqslant1+\limsup\limits_{n\to\infty}\dfrac{\log{\varphi(f,n)}}{\log{n}}=1.$$
\end{proof}

\begin{cor}\label{cor:reg-cur-W}
   Let $f:X\to X$ be a homeomorphism on a regular curve $X$ with at least one wandering point. Then $h_{\mathrm{pol}}(f)=1$.
\end{cor}
\begin{proof} 
One inequality is clear because of the previous corollary. Since $f:X\to X$ possesses a wandering point, then according to Proposition 2.1 in \cite{L}, we have that $h_{\mathrm{pol}}(f)\geqslant1$. We conclude that $h_{\mathrm{pol}}(f)=1$.
\end{proof}

\begin{rem}
  Let $X=\mathbb{S}^1$ and $f:\mathbb{S}^1\to\mathbb{S}^1$ an irrational rotation. Since the map $f$ is an isometry, we have that the polynomial entropy is equal to zero. Note that in this case $\mathrm{NW}(f)=\mathbb{S}^1$, so $f$ possesses no wandering points.
\end{rem}

The following proposition is one of the main results of \cite{GC}. We obtain it as a corollary of Theorem~\ref{kato}.
\begin{prop}
Let $f:I\to I$ be a continuous piecewise monotone map. Let $c_n$ be the number of maximal intervals of monotonicity of the $n-$th iterate $f^n$.  Then 
$$h_{\mathrm{pol}}(f)\leqslant1+\frac{\log{c_n}}{\log{n}}.$$
\end{prop}
\begin{proof}
Note that no two different points from the same interval of monotonicity cannot be mapped to a same point. Therefore, we have that $\varphi(f,n)\leqslant c_n$ and consequently 
$$h_{\mathrm{pol}}(f)\leqslant1+\limsup_{n\to\infty}\dfrac{\log{\varphi(f,n)}}{\log{n}}\leqslant1+\frac{\log{c_n}}{\log{n}}.$$
\end{proof}

We proved the following theorem in \cite{DKL}:

\begin{thrm}\label{(6)}
    Let $X$ be a compact metric space and $f:X\to X$ a homeomorphism such that there exists a wandering point $x_0$. Then $h_{\mathrm{pol}}(F_n(f))\ge n$.
\end{thrm}

When $X$ is a regular curve, we can obtain a stronger result:

\begin{cor}
 Let $f:X\to X$ be a homeomorphism on a regular curve $X$ with at least one wandering point. Then $h_{\mathrm{pol}}(F_n(f))=n$.
\end{cor}

\begin{proof}
 Under these assumptions we have that $h_{\mathrm{pol}}(f)=1$. Since $F_n(f)$ is a factor of $f^{\times n}$, we have that
  $$h_{\mathrm{pol}}(F_n(f))\leqslant h_{\mathrm{pol}}(f^{\times n})=n h_{\mathrm{pol}}(f)=n.$$
 Combining this with Corollary~\ref{(6)}, we obtain that $h_{\mathrm{pol}}(F_n(f))=n$.
\end{proof}

\begin{cor}
 Let $f:X\to X$ be a homeomorphism on a regular curve $X$ with at least one wandering point. Then $h_{\mathrm{pol}}(2^f)=+\infty$.
\end{cor}

\begin{proof}
Note that $F_n(X)$ is a closed and $2^f-$invariant subset of $2^X$ and clearly $F_n(f)=2^f\mid_{F_n(X)}$, so
$$h_{\mathrm{pol}}(2^f)\ge h_{\mathrm{pol}}(F_n(f))=n,$$
for every $n\in\mathbb{N}.$
\end{proof}

\section{Polynomial entropy on local dendrites}

Let $X$ be a local dendrite and $f:X\to X$ a homeomorphism. Since local dendrites are regular curves, we have that $h_{\mathrm{pol}}(f)\le 1$, by Theorem \ref{kato}. If $f$ possesses a wandering point, then $h_{\mathrm{pol}}(f)=1$, as we have seen. Let $f:X\to X$ be a homeomorphism without wandering points. By the main result of \cite{N1}, we have that $\mathrm{NW}(f)=\mathrm{Rec}(f)$ for homeomorphisms on regular curves, so in this case, every point in $X$ is actually recurrent, i.e.\ $f$ is pointwise-recurrent. The following result is proved in \cite{SHR} (see Proposition $1.2$):
\begin{prop}\label{PR=EQ}
Every pointwise-recurrent local dendrite map is equicontinuous.
\end{prop}

Using Proposition \ref{equi} we obtain the following:

\begin{thrm}\label{prop:locden-NW}
 Let $X$ be a local dendrite and $f:X\to X$ a homeomorphism. Then $h_{\mathrm{pol}}(f)\in\{0,1\}$. Moreover $h_{\mathrm{pol}}(f)=1$ if and only if $f$ possesses a wandering point.
\end{thrm}

\begin{cor}
Let $f:X\to X$ be a homeomorphism without wandering points on a local dendrite $X$. Then $h_{\mathrm{pol}}(F_n(f))=0$.
\end{cor}
\begin{proof}
Since 
$$h_{\mathrm{pol}}(F_n(f))\le h_{\mathrm{pol}}(f^{\times n})=nh_{\mathrm{pol}}(f)=0,$$ 
we have that $h_{\mathrm{pol}}(F_n(f))=0$.
\end{proof}

\begin{cor}
Let $f:X\to X$ be a homeomorphism without wandering points on a local dendrite $X$. Then $h_{\mathrm{pol}}(2^f)=0$.
\end{cor}
\begin{proof}
Since $f$ possesses no wandering points, using Proposition \ref{PR=EQ}, we have that $f$ is equicontinuous. It is a known fact that then $2^f$ is also equicontinuous, so we have that $h_{\mathrm{pol}}(2^f)=0$.
\end{proof}

We end this paper with some interesting questions.\\

\textbf{Question1}: Are there any regular curves, which are not local dendrites, such that polynomial entropy of homeomorphisms on them is not equal to $0$ or $1$?

\textbf{Question 2}: Does the upper bound for polynomial entropy we proved to be equal to $1$ for regular curves stands for rational curves?

\end{document}